\documentclass[11pt]{article}

\usepackage[utf8]{inputenc}
\usepackage[T1]{fontenc}
\usepackage{amsthm}
\usepackage{amssymb}
\usepackage{amsmath}
\usepackage{mathtools}
\usepackage{enumerate}
\usepackage{subcaption}
\usepackage[tikz]{bclogo}

\newcommand{\below}{\preccurlyeq}
\newcommand{\sbelow}{\prec}
\newcommand{\vect}{\textbf}

\newcommand{\orl}{\vee}
\newcommand{\andl}{\wedge}
\newcommand{\notl}{\overline}

\newcommand{\splitf}{split}

\newtheorem{theorem}{Theorem}
\newtheorem{lemma}{Lemma}

\newtheorem{corollary}{Corollary}
\newtheorem{claim}{Claim}
\newtheorem{definition}{Definition}

\newtheorem{conjecture}{Conjecture}
\newtheorem{observation}{Observation}

\usepackage{color}

\begin{document}

\author{
  Vadim Lozin\thanks{Mathematics Institute, University of Warwick, UK. E-mail: V.Lozin@warwick.ac.uk.}
  \and
  Igor Razgon\thanks{Department of Computer Science and Information Systems, Birkbeck University of London, UK. E-mail: Igor@dcs.bbk.ac.uk.}     
  \and
  Viktor Zamaraev\thanks{Mathematics Institute, University of Warwick, UK. E-mail: V.Zamaraev@warwick.ac.uk.}
  \and
  Elena Zamaraeva\thanks{Mathematics Institute, University of Warwick, UK. E-mail: E.Zamaraeva@warwick.ac.uk.}
  \and
  Nikolai Yu. Zolotykh\thanks{Institute of Information Technology, Mathematics and Mechanics,
  Lobachevsky State University of Nizhni Novgorod, Russia. E-mail: Nikolai.Zolotykh@itmm.unn.ru.}
}

\title{Specifying a positive threshold function via extremal points}



\maketitle


\begin{abstract}
An extremal point of a positive threshold Boolean function $f$ is either a maximal zero or a minimal one.
It is known that if $f$ depends on all its variables, then the set of its extremal points completely specifies $f$ within the universe 
of threshold functions. However, in some cases, $f$ can be specified by a smaller set. The minimum number of 
points in such a set is the specification number of $f$. 
It was shown in [S.-T.~Hu. Threshold Logic, 1965] that the specification number of a threshold function of 
$n$ variables is at least $n+1$.  
In [M.~Anthony, G.~Brightwell, and J.~Shawe-Taylor. On specifying Boolean functions by labelled examples.
\textit{Discrete Applied Mathematics}, 1995] it was proved that this bound is attained for nested functions and conjectured that for all other threshold functions the specification number is strictly greater than $n+1$.
In the present paper, we resolve this conjecture negatively by 
exhibiting threshold Boolean functions of $n$ variables, which are non-nested and 
for which the specification number is $n+1$. On the other hand, we show that the set of extremal points satisfies 
the statement of the conjecture, i.e.~a positive threshold Boolean function depending on all its $n$ variables has $n+1$ extremal points if and only if
it is nested. To prove this, we reveal an underlying structure of the set of extremal points. 
\end{abstract}

\section{Introduction}

A Boolean function is called a {\it threshold function} (also known as {\it linearly separable} or a {\it halfspace}) 
if there exists a hyperplane separating true and false points of the function.
Threshold functions play fundamental role in the theory of Boolean functions and they appear in a variety of applications such as
electrical engineering, artificial neural networks, reliability theory, game theory etc. (see, for example, \cite{CH11}).

We study the problem of teaching threshold functions in the context of on-line learning with a helpful teacher \cite{GoldmanKearns1995}.
Speaking informally, {\it teaching} an unknown function $f$ in a given class is the problem of producing 
its {\it teaching} (or {\it specifying}) set, i.e. a set of points in the domain which uniquely specifies $f$.
In the present paper, the universe is the set of threshold functions and a specifying set for $f$ is 
a subset $S$ of the points of the Boolean cube such that $f$ is the only threshold function which is consistent with $f$ on $S$.

It is not difficult to see that in the worst case the specifying set contains  all the $2^n$ points of the Boolean cube.
However, in some cases, a threshold function $f$ can be specified by a smaller set, for instance, when $f$ depends on all its variables
and is {\it positive} (or {\it increasing}), i.e. a function where an increase of a variable cannot lead to a decrease of the function. 
In this case, $f$ can be specified by the set of its {\it extremal points}, i.e. its maximal false and minimal true points, of which 
there are at most $\binom{n+1}{\left\lfloor \frac{n+1}{2}\right\rfloor}$ \cite{AnthonyBrightwellShaweTaylor1995}.
Moreover, this description can also be redundant, i.e. sometimes a positive
threshold function $f$ can be specified by a proper subset of its extremal points. 
The minimum cardinality of a teaching set of $f$, i.e. the minimum number of points needed to specify $f$,
is the {\it specification number} of $f$. 
The maximum specification number over all functions in a class is the {\it teaching dimension} of the class.

\cite{Hu1965} showed that the specification number of a threshold function with $n$ variables is at least $n+1$.
\cite{AnthonyBrightwellShaweTaylor1995} proved that this bound is attained for so-called nested functions by showing that positive nested functions contain precisely $n+1$ extremal points.
They also conjectured  that for all other threshold functions 
with $n$ variables the specification number is strictly greater than $n+1$.

\paragraph*{Our contribution}

As our first result, we disprove the conjecture of \cite{AnthonyBrightwellShaweTaylor1995} by showing that for any $n \geq 4$ 
there exist threshold functions with $n$ variables which are non-nested and for which the specification number is $n+1$.

To state our second result, we observe that for positive nested functions the specifying set coincides with the set of extremal points.
This is not the case in our counterexamples to the above conjecture. 
Therefore, our negative resolution of the conjecture leaves open the question on the number of extremal points: 
is it true that for any positive threshold function different 
from nested, the number of extremal points is strictly greater than $n+1$? 
In this paper, we answer this question positively. Moreover, we prove a slightly more general result dealing 
with so-called linear read-once functions, which is an extension of nested functions allowing irrelevant variables (see Section~\ref{sec:pre} for precise definitions). 
More formally, we prove that a positive threshold function $f$ with $k \geq 0$ relevant variables
has exactly $k+1$ extremal points if and only if $f$ is linear read-once. 
Our solution is based on revealing an underlying structure of the set of extremal points.

\paragraph*{Related work}

Upper and lower bounds and the average value for the specification number of a threshold Boolean
function are obtained in \cite{AnthonyBrightwellShaweTaylor1995}.

A number of papers are devoted to the teaching dimension for the class of threshold functions of $k$-valued logic,
i.e. halfspaces defined on the domain $\{0,1,\dots,k-1\}^n$.
Upper bounds for the teaching dimension are obtained in \cite{Hegedus1994,ChirkovZolotykh2016}.
A tight lower bound is stated in \cite{ShevchenkoZolotykh1998}. 
The special case $n=2$ is considered in \cite{AlekseyevBasovaZolotykh,Zamaraeva2016,Zamaraeva2017}.

The problem of teaching is closely related to the problem of learning 
\cite{Angluin1988,AizensteinHegedusHellersteinPitt1998}. 
Learning threshold functions with membership or/and equivalence queries is studied 
by \cite{MaassTuran1994,Hegedus1994,Hegedus1995,ZolotykhShevchenko1997}.
Special case $n = 2$ is considered in \cite{BultmanMaass1995}.
Learning threshold Boolean functions with small weights is investigated in \cite{AbboudAghaBshoutyRadwanSaleh1999,Bshouty2016}.

Teaching or/and learning different classes of read-once (or repetition-free) functions are considered
in \cite{AngluinHellersteinKarpinski1993,BshoutyHancockHellerstein1995,Chistikov2011,ChistikovFedorovaVoronenko2014}.
The importance of linear read-once functions in learning theory is evidenced, in particular, by 
their connection with special types of decision lists \cite{Rivest1987}.

\paragraph*{Organization of the paper} 
All preliminary information related to the paper can be found in Section~\ref{sec:pre}.
The refutation of the conjecture is presented in Section~\ref{sec:cn}. Section~\ref{sec:ex} contains the results about extremal points of a threshold function. 
Section~\ref{sec:con} concludes the paper with a number of open problems.

\section{Preliminaries}
\label{sec:pre}

Let $B = \{ 0, 1 \}$. 
For a \textit{point} $\vect{x} \in B^n$ we denote by $(\vect{x})_i$ the $i$-th
coordinate of $\vect{x}$,
and by $\notl{\vect{x}}$ the point in $B^n$
with $(\notl{\vect{x}})_i = 1$ if and only if $(\vect{x})_i = 0$ for every $i \in [n]$.

Let $f=f(x_1, \ldots, x_n)$ be a Boolean function  on $B^n$. For $k \in [n]$ and $\alpha_k \in \{ 0, 1\}$ 
we  denote by $f_{|x_k=\alpha_k}$ the Boolean function on $B^{n-1}$ defined as
follows:
$$
	f_{|x_k=\alpha_k} (x_1, \ldots, x_{k-1}, x_{k+1}, \ldots, x_n) =
	f(x_1, \ldots, x_{k-1}, \alpha_k, x_{k+1}, \ldots, x_n).
$$
For $i_1, \ldots, i_k \in [n]$ and $\alpha_1, \ldots, \alpha_k \in \{ 0,1 \}$ we denote by
$f_{|x_{i_1} = \alpha_1, \ldots, x_{i_k} = \alpha_k}$ the function 
$(f_{|x_{i_1} = \alpha_1, \ldots, x_{i_{k-1}} = \alpha_{k-1}})_{|x_{i_k} = \alpha_k}$.
We say that $f_{|x_{i_1} = \alpha_1, \ldots, x_{i_k} = \alpha_k}$ is the \textit{restriction} of 
$f$ to $x_{i_1} = \alpha_1, \ldots, x_{i_k} = \alpha_k$.
We also say that a Boolean function $g$ is a \textit{restriction} of a Boolean function $f$
if there exist $i_1, \ldots, i_k \in [n]$ and $\alpha_1, \ldots, \alpha_k \in \{ 0,1 \}$ such that
$g \equiv f_{|x_{i_1} = \alpha_1, \ldots, x_{i_k} = \alpha_k}$, i.e.,
$g(\vect{x}) =  f_{|x_{i_1} = \alpha_1, \ldots, x_{i_k} = \alpha_k}(\vect{x})$ for every $\vect{x} \in B^{n-k}$.

\begin{definition}
A variable $x_k$ is called {\it irrelevant} for $f$ if $f_{|x_k=1} \equiv f_{|x_k=0}$.
Otherwise, $x_k$ is called \textit{relevant} for $f$.
If $x_k$ is irrelevant for $f$ we will also say that $f$ \textit{does not depend on} $x_k$.
\end{definition}

Following the terminology of \cite{CH11}, we say that $\vect{x} \in B^n$ is a {\it true point} of $f$ if $f(\vect{x}) = 1$
and  that $\vect{x} \in B^n$ is a {\it false point} of $f$ if $f(\vect{x}) = 0$.

\subsection{Positive functions and extremal points}

By $\below$ we denote a partial order over the set $B^n$,
induced by inclusion in the power set lattice of the $n$-set.
In other words,  
$\vect{x} \below \vect{y}$ if $(\vect{x})_i = 1$ implies $(\vect{y})_i=1$.
In this case we will say that $\vect{x}$ is \textit{below} $\vect{y}$.
When $\vect{x} \below \vect{y}$ and $\vect{x} \neq \vect{y}$ we will sometimes
write $\vect{x} \sbelow \vect{y}$.

\begin{definition}
A Boolean function $f$ is called \textit{positive monotone} (or simply \textit{positive}) if $f(\vect{x}) = 1$ 
and $\vect{x} \below \vect{y}$ imply $f(\vect{y})=1$.
\end{definition}

For a positive Boolean function $f$, the set of its false points forms a down-set and 
the set of its true points forms an up-set of the partially ordered set $(B^n, \below)$.
We denote by 
\begin{itemize}
\item[$Z^f$] the set of maximal false points,
\item[$U^f$] the set of minimal true points. 
\end{itemize}
We will refer to a point in $Z^f$
as a \textit{maximal zero of $f$} and to a point in $U^f$ as a \textit{minimal one of $f$}.
A point will be called an \textit{extremal point of $f$} if it is either a maximal zero or a minimal one of $f$.
We denote by 
\begin{itemize}
\item[$r(f)$] the number of extremal points of $f$.
\end{itemize}

\subsection{Threshold functions}
\begin{definition}
A Boolean function $f$ on $B^n$ is called a \textit{threshold function} if there exist
$n$ \textit{weights} $w_1, \ldots, w_n \in \mathbb{R}$ and a \textit{threshold} $t \in \mathbb{R}$
such that, for all $(x_1, \ldots, x_n) \in B^n$,
$$
	f(x_1, \ldots, x_n) = 0 \iff \sum\limits_{i=1}^n w_i x_i \leq t.
$$
\end{definition}
The inequality $w_1 x_1 + \ldots + w_n x_n \leq t$ is called \textit{threshold inequality} representing
function $f$.
It is not hard to see that there are uncountably many different threshold inequalities representing a given 
threshold function, and if there exists an inequality with non-negative weights, then $f$ is a positive function.


Let $k \in \mathbb{N}, k \geq 2$. A Boolean function $f$ on $B^n$ is \textit{$k$-summable}
if, for some $r \in \{ 2, \ldots, k \}$, there exist $r$ (not necessarily distinct)
false points $\vect{x}_1, \ldots, \vect{x}_r$ and $r$ (not necessarily distinct) 
true points $\vect{y}_1, \ldots, \vect{y}_r$ such that 
$\sum_{i=1}^r \vect{x}_i = \sum_{i=1}^r \vect{y}_i$ (where the summation is over $\mathbb{R}^n$).
A function is \textit{asummable} if it is not $k$-summable for all $k \geq 2$.

\begin{theorem} \label{th:asum} {\rm {\cite{E61}}}   
	A Boolean function is a threshold function if and only if it is asummable.
\end{theorem}

\subsection{Linear read-once functions and nested functions}

A Boolean function $f$ is called \textit{linear read-once}
if it is either a constant function, or it can be represented by a \textit{nested formula}
defined recursively as follows:
\begin{enumerate}
	\item both literals $x$ and $\notl{x}$ are nested formulas;
	\item $x \orl t$, $x \andl t$, $\notl{x} \orl t$, $\notl{x} \andl t$ are nested formulas,
	where $x$ is a variable and $t$ is a nested formula that contains neither $x$, nor $\notl{x}$.
\end{enumerate}

\cite{lists} showed that the class of linear read-once functions is precisely the intersection 
of threshold and read-once functions. 


A linear read-once function is called {\it nested} if it depends on all its variables.
For example, the function $(x_1 \orl x_2) x_3 x_5$ considered as a function of 5 variables $x_1, \ldots, x_5$
is linear read-once, but not nested, since $x_4$ is an irrelevant variable. If this function is considered
as a function of 4 variables $x_1, x_2, x_3, x_5$, then all its variables are relevant and therefore
the function is also nested.



It is not difficult to see that a linear read-once function $f$ is positive if and only if 
a nested formula representing $f$ does not contain negations.

\subsection{Specifying sets and specification number}

Let $\mathcal{F}$ be a class of Boolean functions of $n$ variables, and let $f \in \mathcal{F}$. 

\begin{definition}
A set of points $S \subseteq B^n$ is a \textit{specifying set for} $f$ in $\mathcal{F}$
if the only function in $\mathcal{F}$ consistent with $f$ on $S$ is $f$ itself. 
In this case we also say that $S$ \textit{specifies} $f$ in the class  $\mathcal{F}$.
The minimal cardinality of specifying set for $f$ in $\mathcal{F}$ is called the 
\textit{specification number of $f$} (in $\mathcal{F}$) and denoted $\sigma_{\mathcal{F}}(f)$.
\end{definition}
 
Let $\mathcal{H}_n$ be the class of threshold Boolean functions of $n$ variables. 
\cite{Hu1965} and later \cite{AnthonyBrightwellShaweTaylor1995} showed that the specification number of a threshold function of $n$ variables is at least $n+1$.  

\begin{theorem}\label{th:th-spec}{\rm \cite{Hu1965,AnthonyBrightwellShaweTaylor1995}}
	For any threshold Boolean function $f$ of $n$ variables $\sigma_{\mathcal{H}_n}(f) \geq n+1$. 
\end{theorem}

It was also shown in \cite{AnthonyBrightwellShaweTaylor1995} that the nested functions attain the lower bound.

\begin{theorem}\label{th:nested-spec}{\rm \cite{AnthonyBrightwellShaweTaylor1995}}
	For any nested function $f$ of $n$ variables $\sigma_{\mathcal{H}_n}(f) = n+1$.
\end{theorem}


\subsection{Essential points}

In estimating the specification number of a threshold Boolean function $f \in \mathcal{H}_n$ 
it is often useful to consider essential points of $f$ defined as follows.

\begin{definition} 
A point $\vect{x}$ is \textit{essential} for $f$ (with respect to class $\mathcal{H}_n$), if
there exists a function $g \in \mathcal{H}_n$ such that $g(\vect{x}) \neq f(\vect{x})$ and
$g(\vect{y}) = f(\vect{y})$ for every $\vect{y} \in B^n$, $\vect{y} \neq \vect{x}$.
\end{definition}
 
Clearly, any specifying set for $f$ must contain all essential points for $f$. It turns out that 
the essential points alone are sufficient to specify $f$ in $\mathcal{H}_n$ \cite{C65}.
Therefore, we have the following well-known result.
\begin{theorem}\label{cl:ess-sigma}{\rm \cite{C65}}
	The specification number $\sigma_{\mathcal{H}_n}(f)$ of a function $f \in \mathcal{H}_n$
	is equal to the number of essential points of $f$.
\end{theorem}

\subsection{The number of essential points versus the number of extremal points}

It was observed in \cite{AnthonyBrightwellShaweTaylor1995} that in the study of specification number of threshold functions, one can be restricted to positive functions. 
To prove Theorem~\ref{th:nested-spec}, \cite{AnthonyBrightwellShaweTaylor1995} first showed that for a positive
threshold function $f$, which depends on all its variables, the set $Z^f \cup U^f$ of extremal points
specifies $f$. Then they proved that for any positive nested function $f$ of $n$ variables
$|Z^f \cup U^f| = n+1$.

In addition to proving Theorem~\ref{th:nested-spec}, \cite{AnthonyBrightwellShaweTaylor1995} also conjectured
that nested functions are the only functions with the specification number $n+1$ in the class 
$\mathcal{H}_n$.

\begin{conjecture}\label{con:conjecture}{\rm\cite{AnthonyBrightwellShaweTaylor1995}}
	If $f \in \mathcal{H}_n$ has the specification number $n+1$, then $f$ is nested.
\end{conjecture}

In the present paper, we disprove Conjecture \ref{con:conjecture} by demonstrating
for every $n \geq 4$ a threshold non-nested function of $n$ variables with the specification
number $n+1$.

On the other hand, we show that the conjecture becomes a true statement if
we replace `specification number' by `number of extremal points'. In fact, we prove a more general
result saying that a positive threshold function $f$ with $k$ relevant variables is linear read-once if and only if it has exactly $k+1$ extremal
points.
For this purpose, the following special type of functions appears to be technically useful.

\begin{definition}
We say that a Boolean function $f = f(x_1, \ldots, x_n)$ is \textit{\splitf{}} if there exists
$i \in [n]$ such that $f_{|x_i=0} \equiv \vect{0}$ or $f_{|x_i=1} \equiv \vect{1}$.
\end{definition}

In what follows, we will need the next two observations, which can be easily verified.

\begin{observation}\label{obs:lro-split}
	Any positive linear read-once function is split.
\end{observation}

\begin{observation}\label{obs:lro-restriciton}
	Any restriction of a linear read-once function is also linear read-once.	
\end{observation}

\section{Non-nested functions with small specification number}
\label{sec:cn}


In this section, we disprove Conjecture~\ref{con:conjecture}. To this end, we show in the following theorem 
that the minimum value of the specification number is attained in the class of threshold functions not only by nested functions. 

\begin{theorem}\label{prop:1}
	For a natural number $n$, $n \geq 4$ let $f_n = f(x_1, \ldots, x_n)$ be a function defined by its DNF
	$$
		x_1 x_2 \vee x_1 x_3 \vee \dots \vee x_1x_{n-1} \vee x_2x_3\dots x_n.
	$$
	Then $f_n$ is positive, not linear read-once, threshold function, depending on all its variables,
	 and the specification number of $f_n$ is $n+1$.
\end{theorem}
\begin{proof}
	Clearly, $f_n$ depends on all its variables.
	Furthermore, $f_n$ is positive, since its DNF contains no negation of a variable.
	Also, it is easy to verify that $f$ is not split, and therefore by Observation~\ref{obs:lro-split}  $f$ is not linear read-once.
	
	Now, we claim that the CNF of $f_n$ is
	$$
		(x_1 \vee x_2)(x_1 \vee x_3)\dots (x_1 \vee x_n)(x_2 \vee x_3 \vee \dots \vee x_{n-1}).
	$$
	Indeed, the equivalence of the DNF and CNF can be directly checked by expanding the latter
	and applying the absorption law:
	\begin{equation*}
		\begin{split}
		       & (x_1 \vee x_2)(x_1 \vee x_3)\dots (x_1 \vee x_n)(x_2 \vee x_3 \vee \dots \vee x_{n-1}) \\
			& = (x_1 \vee x_2x_3\dots x_n)(x_2 \vee x_3 \vee \dots \vee x_{n-1}) \\
			& = x_1 x_2 \vee x_1 x_3 \vee \dots \vee x_1x_{n-1} \vee x_2x_3\dots x_n.
		\end{split}
	\end{equation*}
	
	\noindent
	From the DNF and the CNF of $f_n$ we retrieve the minimal ones
	$$
	\begin{array}{r}
		\vect{x}_1 = (1,1,0,\dots,0,0),\\
		\vect{x}_2 = (1,0,1,\dots,0,0),\\
		\dotfill                         \\
		\vect{x}_{n-2} = (1,0,0,\dots,1,0),\\
		\vect{x}_{n-1} = (0,1,1,\dots,1,1)\phantom{,}\\
	\end{array}	
	$$
	and maximal zeros of $f_n$
	$$
	\begin{array}{r}
    		\vect{y}_1 = (0,0,1,\dots,1,1),\\
    		\vect{y}_2 = (0,1,0,\dots,1,1),\\
		\dotfill                             \\
		\vect{y}_{n-2} = (0,1,1,\dots,0,1),\\
		\vect{z}_1 = (0,1,1,\dots,1,0),\\
		\vect{z}_2 = (1,0,0,\dots,0,1),\\
	\end{array}
	$$
	respectively (see Theorems 1.26, 1.27 in \cite{CH11}).
	It is easy to check that all minimal ones $\vect{x}_1,\vect{x}_2,\dots,\vect{x}_{n-1}$ 
	satisfy the equation 
	$$
		(2n-5)x_1 + 2(x_2+x_3+\dots+x_{n-1})+x_n = 2n-3,
	$$
	and all maximal zeros $\vect{y}_1, \vect{y}_2, \dots, \vect{y}_{n-2}, \vect{z}_1, \vect{z}_2$
	satisfy the inequality
	$$
		(2n-5)x_1 + 2(x_2+x_3+\dots+x_{n-1})+x_n \leq 2n-4.
	$$
	Hence the latter is a threshold inequality representing $f_n$.
	
	Since for a positive threshold function $f$ which depends on all its variables the set of
 	extremal points specifies $f$, and every essential point of $f$ must belong to each specifying set,
 	we conclude that every essential point of $f_n$ is extremal.
	
	
	Let us show that the points $\vect{y}_1, \vect{y}_2, \dots, \vect{y}_{n-2}$ are not essential 
	for $f_n$. Suppose to the contrary that there exists a threshold function $g_i$ that differs from $f_n$ only 
	in the point $\vect{y}_i$, $i \in [n-2]$, i.e., $g_i(\vect{y}_i)=1$ and $g_i(\vect{x})=f_n(\vect{x})$ 
	for every $\vect{x}\ne \vect{y}_i$.
	Then $\vect{x}_i + \vect{y}_i = \vect{z}_1 + \vect{z}_2$, and hence $g_i$ is 2-summable. 
	Therefore by Theorem~\ref{th:asum} function $g_i$ is not threshold. A contradiction.

	The above discussion together with Theorems \ref{th:th-spec} and \ref{cl:ess-sigma} imply that all
	the remaining $n+1$ extremal points $\vect{x}_1,\vect{x}_2,\dots,\vect{x}_{n-1},\vect{z}_1,\vect{z}_2$
	are essential, and therefore $\sigma_{\mathcal{H}_n}(f_n) = n+1$.
\end{proof}

\section{Extremal points of a threshold function}
\label{sec:ex}

The main goal of this section is to prove the following theorem.
\begin{theorem}\label{th:extremal_main}
	Let $f = f(x_1, \ldots, x_n)$ be a positive threshold function with $k \geq 0$ relevant variables. 
	Then the number of extremal points of $f$ is at least $k+1$. Moreover $f$
	has exactly $k+1$ extremal points if and only if $f$ is linear read-once.
\end{theorem}


We will prove Theorem~\ref{th:extremal_main} by induction on $n$. The statement
is easily verifiable for $n=1$. Let $n > 1$ and assume that the theorem is true for
functions of at most $n-1$ variables. In the rest of the section we prove
the statement for $n$-variable functions. Our strategy consists of three major steps.
First, we prove the statement for split functions in Section~\ref{sec:split}. 
This case includes linear read-once functions.
Then, in Section~\ref{sec:ns1}, we prove the result 
for non-split functions $f$ which have a variable $x_i$ such that both restrictions 
$f_{|x_i=0}$ and $f_{|x_i=1}$ are split. Finally, in Section~\ref{sec:ns2}, we consider the case of non-split
functions $f$, where for every variable $x_i$ of $f$ at least one of the restrictions 
$f_{|x_i=0}$ and $f_{|x_i=1}$ is non-split. 
In this case, the proof is based on a structural characterization of the set 
of extremal points, which is of independent interest and which is presented in Section~\ref{sec:str}.


\subsection{The structure of the set of extremal points}
\label{sec:str}

We say that a maximal zero (\textit{resp.} minimal one) $\vect{y}$ of $f(x_1, \ldots, x_n)$
\textit{corresponds to a variable} $x_i$ if $(\vect{y})_i = 0$ (\textit{resp.} $(\vect{y})_i = 1$).
%
A pair $(\vect{a}, \vect{b})$ of points in $B^n$ is called \textit{$x_i$-extremal for $f$} if
\begin{enumerate}
	\item $\vect{a}$ is a maximal zero of $f$ corresponding to $x_i$;
	\item $\vect{b}$ is a minimal one of $f$ corresponding to $x_i$; and
	\item $(\vect{a})_j \geq (\vect{b})_j$ for every $j \in [n] \setminus \{ i \}$.
\end{enumerate}

\begin{claim}\label{cl:extremal_pair}
	Let $f$ be a positive function and $i \in [n]$.
	Then 
	\begin{enumerate}
		\item for every maximal zero $\vect{\textup{a}}$ of $f$ corresponding to $x_i$ there exists 
		a minimal one $\vect{\textup{b}}$ of $f$ corresponding to $x_i$
		such that $(\vect{\textup{a}}, \vect{\textup{b}})$ is an $x_i$-extremal pair for $f$;
		\item for every minimal one $\vect{\textup{b}}$ of $f$ corresponding to $x_i$ there exists 
		a maximal zero $\vect{\textup{a}}$ of $f$ corresponding to $x_i$
		such that $(\vect{\textup{a}}, \vect{\textup{b}})$ is an $x_i$-extremal pair for $f$.
	\end{enumerate}
\end{claim}
\begin{proof}
We prove the first part of the claim, the second part can be proved similarly.
Consider a maximal zero $\vect{a}$ of $f$ corresponding to $x_i$ and the vector $\vect{b}'$ such that $(\vect{a})_j=(\vect{b}')_j$ for all $j \neq i$ and $(\vect{b}')_i=1$.
Since $\vect{a} \sbelow \vect{b}'$ and $\vect{a}$ is a maximal zero, we have $f(\vect{b}') = 1$.
Let $\vect{b}$ be a minimal one of $f$ such that $\vect{b} \below \vect{b}'$. 
Then $(\vect{b})_i=1$ for otherwise $\vect{b}$ would be below $\vect{a}$, which in turn would contradict
positivity of $f$. 
Now since $\vect{a}$ and $\vect{b}'$ differ only in coordinate $i$ and $\vect{b} \below \vect{b}'$, we conclude that $(\vect{a})_j \geq (\vect{b})_j$ for every $j \in [n] \setminus \{ i \}$, and therefore 
$(\vect{a},\vect{b})$ is an $x_i$-extremal pair for $f$.
\end{proof}

Let $g=g(y_1, \ldots, y_n)$ be a positive function, and let $\{ y_{i_1}, \ldots, y_{i_k} \}$
be a subset of the relevant variables of $g$. For every variable $y_{i_j}$, $j \in [k]$ we fix 
an $y_{i_j}$-extremal pair $(\vect{a}_{i_j}, \vect{b}_{i_j})$. Now we define a graph 
$H(g, y_{i_1}, \ldots, y_{i_k})$ as an undirected graph with vertex set 
$\{ \vect{a}_{i_j}, \vect{b}_{i_j}~|~j \in [k] \}$ and edge set $\{ \{\vect{a}_{i_j}, \vect{b}_{i_j}\}~|~j \in [k] \}$. 
We call $H(g, y_{i_1}, \ldots, y_{i_k})$ an extremal graph and observe that this graph is defined not uniquely.

\begin{lemma}\label{lem:acyclic}
	If $g$ is a threshold function, then $H = H(g, y_{i_1}, \ldots, y_{i_k})$ is an acyclic graph. 
\end{lemma}
\begin{proof}
It follows from the definitions of an $x_i$-extremal pair and of an extremal graph that
$H$ does not have multiple edges and that $H$ is a bipartite graph with parts $A = \{ \vect{a}_{i_j}~|~j \in [k] \}$
and $B = \{ \vect{b}_{i_j}~|~j \in [k] \}$.
Suppose to the contrary that $H$ has a cycle of length $2r$, for some $r \in \{ 2, \ldots, k \}$. 
Let $R$ and $Q$ be the sets of vertices of the cycle belonging to $A$ and $B$, respectively.
For $i \in [n]$ and $\alpha \in \{ 0, 1 \}$ we denote by $R_{\alpha}^i$ the set of vertices $\vect{y} \in R$
with $(\vect{y})_i = \alpha$. Similarly, $Q_{\alpha}^i$ denotes the set of vertices 
$\vect{y} \in Q$ with $(\vect{y})_i = \alpha$.

Fix an index $i \in [n]$. By definition of an $x_i$-extremal pair and of an extremal graph, there is at most one edge between the vertices of $Q_{1}^i$
and the vertices of $R_{0}^i$. Therefore, the number $2|Q_{1}^i|$ of the edges in the cycle incident to the vertices in $Q_{1}^i$ is at most one
more than the number $2|R_{1}^i|$ of the edges incident to the vertices in $R_{1}^i$. This implies that
$|Q_{1}^i| \leq |R_{1}^i|$. If this inequality is strict, we modify the set $Q$ by choosing 
arbitrarily $|R_{1}^{i}| - |Q_{1}^{i}|$ points in $Q_{0}^i$ and changing their $i$-th coordinates
from $0$ to $1$. Since $g$ is positive, the modified points remain true points for $g$. 

Applying this procedure for each $i \in [n]$, we obtain the set $R$ of false points and the set $Q$ of 
true points both of size $r$
such that $|Q_{1}^i| = |R_{1}^i|$ for all $i$. Therefore, 
$\sum_{\vect{x} \in R} \vect{x} = \sum_{\vect{y} \in Q} \vect{y}$, showing that
$g$ is $k$-summable. Hence, by Theorem \ref{th:asum}, $g$ is not threshold, which contradicts the assumption of the lemma.
%
%
\end{proof}



\subsection{Split functions}
\label{sec:split}

\begin{lemma}\label{lem:split}
	Let $f = f(x_1, \ldots, x_n)$ be a positive threshold split function 
	with $k \geq 0$ relevant variables. 
	Then the number of extremal points of $f$ is at least $k+1$. Moreover $f$
	has exactly $k+1$ extremal points if and only if $f$ is linear read-once.
\end{lemma}
\begin{proof}
	The case $k=0$ is trivial, and therefore we assume that $k \geq 1$.

	Let $x_i$ be a variable of $f$ such that $f_{|x_i = 0} \equiv \vect{0}$
	(the case $f_{|x_i = 1} \equiv \vect{1}$ is similar). Let $f_0 = f_{|x_i = 0}$ and $f_1 = f_{|x_i = 1}$.
	Clearly, $x_i$ is a relevant variable of $f$, otherwise $f \equiv \vect{0}$, that is, $k=0$.
	Since every relevant variable of $f$ is relevant for at least one of the functions
	$f_0$ and $f_1$, we conclude that $f_1$ has $k-1$ relevant variables.
	
	The equivalence $f_0 \equiv \vect{0}$ implies that for every extremal point 
	$(\alpha_1, \ldots, \alpha_{i-1}, \alpha_{i+1}, \ldots, \alpha_n)$ of $f_1$, the corresponding
	point $(\alpha_1, \ldots, \alpha_{i-1}, 1, \alpha_{i+1}, \ldots, \alpha_n)$ is extremal for $f$.
	For the same reason, there is only one extremal point of $f$ with the $i$-th coordinate being
	equal to zero, namely, the point with all coordinates equal to one, except for the $i$-th coordinate.
	Hence, $r(f) = r(f_1) + 1$. 
	\begin{enumerate}
		\item If $f_1$ is linear read-once, then $f$ is also linear read-once, since 
		$f$ can be expressed as $x_i \andl f_1$. 
		By the induction hypothesis $r(f_1) = k$ and therefore $r(f) = k+1$.
		
		\item If $f_1$ is not linear read-once, then from Observation \ref{obs:lro-restriciton}
		we conclude that $f$ is also not linear read-once.
		By the induction hypothesis $r(f_1) > k$ and therefore $r(f) > k+1$.
	\end{enumerate}
\end{proof}

\subsection{Non-split functions with split restrictions}
\label{sec:ns1}

\begin{claim}\label{cl:com_var}
	Let $f = f(x_1, \ldots, x_n)$ be a positive threshold non-split function. 
	If there exists $i \in [n]$ such that both $f_0 = f_{|x_i=0}$ and $f_1 = f_{|x_i=1}$ are split,
	then there exists $s \in [n] \setminus \{ i \}$ such that
	${f_0}_{|x_s = 0} \equiv \vect{\textup{0}}$ and ${f_1}_{|x_s = 1} \equiv \vect{\textup{1}}$.
\end{claim}
\begin{proof}
	Since $f_0$ is split, there exists $p \in [n]$ such that ${f_0}_{|x_p=0} \equiv \vect{0}$ or 
	${f_0}_{|x_p=1} \equiv \vect{1}$. We claim that the latter case is impossible. 
	Indeed, as ${f_0}_{|x_p=1} = f_{|x_i=0,x_p=1}$, positivity of $f$ 
	and ${f_0}_{|x_p=1} \equiv \vect{1}$ imply $f_{|x_i=1,x_p=1} \equiv \vect{1}$, and therefore 
	$f_{|x_p=1} \equiv \vect{1}$.
	This contradicts the assumption that $f$ is non-split. 
	Hence, ${f_0}_{|x_p=0} \equiv \vect{0}$.
	Similarly, one can show that ${f_1}_{|x_r=1} \equiv \vect{1}$ for some $r \in [n]$. 
	If $p = r$, then we are done. 
	
	Assume that $p \neq r$. Let $\vect{a}$ be the point in $B^n$ that has exactly two 1's in
	coordinates $i$ and $p$. If $f(\vect{a}) = 1$, then by positivity
	$f_{|x_i=1,x_p=1} = {f_1}_{|x_p=1} \equiv 1$, and the claim follows for $s = p$.
	Let now $\vect{b}$ be a point in $B^n$ that has exactly two 0's in coordinates $i$ and $r$.
	If $f(\vect{b}) = 0$, then by positivity $f_{|x_i=0,x_r=0} = {f_0}_{|x_r=0} \equiv 0$, 
	and the claim follows for $s = r$. 
	
	Assume now that $f(\vect{a}) = 0$ and $f(\vect{b}) = 1$. 
	Since ${f_0}_{|x_p=0} \equiv \vect{0}$ and ${f_1}_{|x_r=1} \equiv \vect{1}$
	we conclude that 
	$f(\notl{\vect{a}})=0$ and $f(\notl{\vect{b}})=1$. Therefore, $\vect{a}+\notl{\vect{a}}=\vect{b}+\notl{\vect{b}}$ and hence by Theorem~\ref{th:asum} 
	$f$ is not threshold. This contradiction completes the proof.
\end{proof}

\begin{corollary}\label{cor:non-split}
	\begin{enumerate}
		\item[(a)] Variable $x_s$ from Claim \ref{cl:com_var} is relevant for
		both functions $f_0$ and $f_1$.
		
		\item[(b)] 
		If a point 
		$\vect{\textup{a}} = (\alpha_1, \ldots, \alpha_{i-1}, \alpha_{i+1}, \ldots, \alpha_{n}) \in B^{n-1}$ 
		is an extremal point of $f_{\alpha_i}$, $\alpha_i \in \{ 0, 1\}$,
		then 
		$\vect{\textup{a}}' = (\alpha_1, \ldots, \alpha_{i-1}, \alpha_i, \alpha_{i+1}, \ldots, \alpha_{n-1}) \in B^n$
		is an extremal point of $f$.
%
	\end{enumerate}
\end{corollary}
\begin{proof}
	\begin{enumerate}
		\item[(a)] Suppose to the contrary that $f_0$ does not depend on $x_s$. Then
		${f_0}_{|x_s=1} \equiv {f_0}_{|x_s=0} \equiv \vect{0}$, and therefore
		$f_0 = f_{x_i=0} \equiv 0$, which contradicts the assumption that $f$ is non-split.
		Similarly, one can show that $x_s$ is relevant for $f_1$.
		

		\item[(b)] We prove the statement for $\alpha_i = 1$. For $\alpha_i = 0$ the arguments 
		are symmetric.
		If $\vect{a}$ is a maximal zero of $f_1$, then $\vect{a}'$ is a maximal zero of $f$.
		Indeed, for every point 
		$\vect{b}' = (\beta_1, \ldots, \beta_{i-1}, \beta_i, \beta_{i+1}, \ldots, \beta_{n}) \in B^n$ such that
		$\vect{a}' \sbelow \vect{b}'$ 
		we have $\beta_i = 1$. Hence 
		$\vect{a} \sbelow \vect{b} = (\beta_1, \ldots, \beta_{i-1}, \beta_{i+1}, \ldots, \beta_{n})$,
		and $f_1(\vect{b}) = f(\vect{b}')$. Therefore $f(\vect{b}') = 0$ would imply that $\vect{a}$
		is not a maximal zero of $f_1$. This contradiction shows that $\vect{a}'$ is a maximal zero of $f$.
		
		Let now $\vect{a}$ be a minimal one of $f_1$. For convenience, without loss of generality, 
		we assume that $s < i$. Suppose to the contrary, that $\vect{a}'$ is not a minimal one of $f$, i.e., 
		there exists a point
		$\vect{b}' = (\beta_1, \ldots, \beta_{i-1}, \beta_i, \beta_{i+1}, \ldots, \beta_{n}) \in B^n$
		such that  $\vect{b}'\sbelow \vect{a}'$ and $f(\vect{b}') = 1$. Note that if $\beta_i = 1$, then
		$\vect{b} \sbelow \vect{a}$ and $f(\vect{b}') = f_1(\vect{b})$, where as before,
		$\vect{b} = (\beta_1, \ldots, \beta_{i-1}, \beta_{i+1}, \ldots, \beta_{n})$.  
		Since $\vect{a}$ is a minimal one of $f_1$, we conclude that $f_1(\vect{b}) = f(\vect{b}') = 0$, 
		which is a contradiction. 
		Therefore we assume further that $\beta_i = 0$ and
		distinguish between two cases: 
		\begin{enumerate}
			\item[$\beta_s = 0$.] In this case 
			$$
				f(\vect{b}') = (\notl{\beta_i}  \andl f_0(\vect{b})) \orl (\beta_i  \andl f_1(\vect{b}))
				= f_0(\vect{b}) = 0,
			$$
			where the latter equality follows from ${f_0}_{|x_s = 0} \equiv \vect{0}$.
			This is a contradiction to our assumption that $f(\vect{b}') = 1$.
			
			\item[$\beta_s = 1$.] In this case, $\alpha_s = 1$. 
			Note that the equivalence ${f_1}_{|x_s = 1} \equiv \vect{1}$ means that function $f_1$
			takes value 1 on every point with $s$-th coordinate being equal to 1. Together with
			the minimality of $\vect{a}$ this implies that the only non-zero component of $\vect{a}$ is 
			$\alpha_s$.
			Hence, the only non-zero component of $\vect{b}'$ is $\beta_s$.
			Therefore $f(\vect{b}') = 1$ and positivity of $f$ imply
			 $f_{|x_s = 1} \equiv \vect{1}$, which contradicts the assumption that $f$ is
			 non-split.
		\end{enumerate}
%
	\end{enumerate}
\end{proof}

\begin{lemma}\label{lem:non-split-split-projection}
	Let $f = f(x_1, \ldots, x_n)$ be a positive threshold non-split function 
	with $k$ relevant variables, and there exists $i \in [n]$ such that
	both $f_0 = f_{|x_i=0}$ and $f_1 = f_{|x_i=1}$ are split.
	Then the number of extremal points of $f$ is at least $k+2$.
\end{lemma}
\begin{proof}
	Let $s \in [n] \setminus \{ i \}$ be an index guaranteed by Claim \ref{cl:com_var}.
	Let $P, P_0$, and $P_1$ be the sets of relevant variables of $f, f_0$, and $f_1$,
	respectively. Since any relevant variable of $f$ is a relevant variable of at least one 
	of the functions $f_0, f_1$ and, by Corollary \ref{cor:non-split} (a), $x_s$ is a relevant 
	variable of both of them, we have
	$$
		k = |P| \leq |P_0 \cup P_1| + 1 = |P_0| + |P_1| - |P_0 \cap P_1| + 1 \leq |P_0| + |P_1|.
	$$
	By the induction hypothesis, $r(f_i) \geq |P_i|+1$, where $i = 0,1$. 
	Finally, by Corollary \ref{cor:non-split} (b) the number $r(f)$ of 
	extremal points of $f$ is at least $r(f_0) + r(f_1) \geq |P_0| + |P_1| + 2 \geq k + 2$.
\end{proof}

\subsection{Non-split functions without split restrictions}
\label{sec:ns2}

Due to Lemmas \ref{lem:split} and \ref{lem:non-split-split-projection}  
it remains to show the bound for a positive threshold non-split function $f = f(x_1, \ldots, x_n)$ 
such that for every $i \in [n]$ at least one of $f_0 = f_{|x_i=0}$ and $f_1 = f_{|x_i=1}$ is non-split.

Assume without loss of generality that $x_n$ is a relevant variable of $f$, and
let $f_0 = f_{|x_n=0}$ and $f_1 = f_{|x_n=1}$. 
We assume that $f_0$ is non-split and prove that $f$ has at least $k+2$ extremal points, where
$k$ is the number of relevant variables of $f$.
The case when $f_0$ is split, but $f_1$ is non-split is proved similarly.
Let us denote the number of relevant variables of $f_0$ by $m$. Clearly, $1 \leq m \leq k-1$.
Exactly $k-1-m$ of $k$ relevant variables of $f$ became irrelevant for the function $f_0$.
Note that these $k-1-m$ variables are necessary relevant for the function $f_1$.
By the induction hypothesis, the number $r(f_0)$ of extremal points of $f_0$ is at least $m+2$.

We introduce the following notation:
\begin{enumerate}
	\item[$C_0$] -- the set of maximal zeros of $f$ corresponding to $x_n$;
	\item[$P_0$] -- the set of all other maximal zeros of $f$, i.e., $P_0 = Z^f \setminus C_0$;
	\item[$C_1$] -- the set of minimal ones of $f$ corresponding to $x_n$;
	\item[$P_1$] -- the set of all other minimal ones of $f$, i.e., $P_1 = U^f \setminus C_1$.
\end{enumerate}

For a set $A \subseteq B^n$ we will denote by $A^*$ the restriction of $A$ into the first $n-1$
coordinates, i.e., 
$A^* = \{ (\alpha_1, \ldots, \alpha_{n-1})~|~(\alpha_1, \ldots, \alpha_{n-1},\alpha_{n}) \in A \text{ for some }
\alpha_n \in \{0,1\} \}$.

By definition, the number of extremal points of $f$ is 
\begin{equation}\label{eq:r_f}
	r(f) = |C_0| + |P_1| + |C_1| + |P_0| = |C_0^*| + |P_1^*| + |C_1^*| + |P_0^*|.
\end{equation}

We want to express $r(f)$ in terms of the number of extremal points of $f_0$ and $f_1$. For this
we need several observations. First, for every extremal point 
$(\alpha_1, \ldots, \alpha_{n-1}, \alpha_n)$ for $f$ the point $(\alpha_1, \ldots, \alpha_{n-1})$
is extremal for $f_{\alpha_n}$. Furthermore, we have the following straightforward claim.

\begin{claim}\label{cl:p1_p0}
	$P_1^*$ is the set of minimal ones of $f_0$ and $P_0^*$ is the set of maximal zeros of $f_1$.
\end{claim}

In contrast to minimal ones of $f_0$, the set of maximal zeros of $f_0$ in addition to the points in $C_0^*$
may contain extra points, which we denote by $N_0^*$. In other words, $Z^{f_0} = C_0^* \cup N_0^*$.
Similarly, besides $C_1^*$, the set of minimal ones of $f_1$ may contain additional points, which we denote
by $N_1^*$. That is, $U^{f_0} = C_1^* \cup N_1^*$.

\begin{claim}\label{cl:n0_p0}
	The set $N_0^*$ is a subset of the set $P_0^*$ of maximal zeros of $f_1$.
	The set $N_1^*$ is a subset of the set $P_1^*$ of minimal ones of $f_0$.
\end{claim}
\begin{proof}
We will prove the first part of the statement, the second one is proved similarly.
Suppose to the contrary that there exists a point 
$\vect{a} = (\alpha_1, \ldots, \alpha_{n-1}) \in N_0^* \setminus P_0^*$, which is a maximal zero for
$f_0$, but is not a maximal zero for $f_1$.
Notice that $f_1(\vect{a}) = 0$, as otherwise $(\alpha_1, \ldots, \alpha_{n-1}, 0)$ would be a maximal zero for
$f$, which is not the case, since $\vect{a} \notin C_0^*$.
Since $\vect{a}$ is not a maximal zero for $f_1$, there exists a maximal zero $\vect{b} \in B^{n-1}$
for $f_1$ such that $\vect{a} \sbelow \vect{b}$.
But then we have $f_0(\vect{b}) = 1$ and $f_1(\vect{b}) = 0$, which contradicts positivity
of function $f$.
\end{proof}

From Claim \ref{cl:p1_p0} we have 
$r(f_0) = |Z^{f_0} \cup U^{f_0}| = |C_0^*| + |N_0^*| + |P_1^*|$, which together with (\ref{eq:r_f}) 
and Claim \ref{cl:n0_p0} imply

\begin{equation}\label{eq:rf}
\begin{split}
	r(f) & = |C_0^*| + |P_1^*| + |C_1^*| + |P_0^*| = |C_0^*| + |P_1^*| + |C_1^*|  + |N_0^*| + 
	|P_0^* \setminus N_0^*| \\ 
	& = r(f_0) + |C_1^*| + |P_0^* \setminus N_0^*|.
\end{split}
\end{equation}

Using the induction hypothesis we conclude that $r(f) \geq m+2 + |C_1^*| + |P_0^* \setminus N_0^*|$.
To derive the desired bound $r(f) \geq k+2$, in the rest of this section we show that 
$C_1^* \cup P_0^* \setminus N_0^*$ contains at least $k-m$ points.

\begin{claim}\label{cl:proper_pairs}
	Let $x_i$, $i \in [n-1]$, be a relevant variable for $f_1$, but irrelevant for $f_0$.
	Then there exists an $x_i$-extremal pair $(\vect{\textup{a}}, \vect{\textup{b}})$ for $f_1$ such that 
	$\vect{\textup{a}} \in P_0^* \setminus N_0^*$ and $\vect{\textup{b}} \in C_1^*$.
\end{claim}
\begin{proof}
First, let us show that an $x_i$-extremal pair always exists.
Since $x_i$ is relevant for $f_1$, there exists a pair of points $\vect{x}$ and $\vect{y}$,
which differ only in the $i$-th coordinate and $f_1(\vect{x}) \neq f_1(\vect{y})$. Without loss of generality, 
let $(\vect{x})_i = 0$ and $(\vect{y})_i = 1$. Then by positivity, $f_1(\vect{x}) = 0$
and $f_1(\vect{y}) = 1$. Let $\vect{x}'$ be any maximal zero of $f_1$ such that $\vect{x} \below \vect{x}'$.
Then obviously $\vect{x}'$ is a maximal zero corresponding to $x_i$ and the existence of 
an $x_i$-extremal pair for $f_1$ follows from Claim \ref{cl:extremal_pair}.
%

We claim that $(\vect{x})_i = 1$ for every $\vect{x} \in N_0^*$. Indeed, if $(\vect{x})_i = 0$
for a maximal zero $\vect{x} \in N_0^*$, then changing in $\vect{x}$ the $i$-th coordinate from $0$
to $1$ we would obtain the point $\vect{x}'$ with $f_0(\vect{x}') = 1 \neq f_0(\vect{x})$, which
would contradict the assumption that $x_i$ is irrelevant for $f_0$.
Similarly, one can show that $(\vect{y})_i = 0$ for every $\vect{y} \in P_1^*$.

The above observations together with Claim \ref{cl:n0_p0} imply that 
every maximal zero for $f_1$ corresponding to $x_i$ belongs to $P_0^* \setminus N_0^*$ 
and every minimal one for $f_1$ corresponding to $x_i$ belongs to $C_1^*$. Hence the claim.
\end{proof}

Recall that there are exactly $s = k - 1 - m$ variables that are relevant for $f_1$ and irrelevant for $f_0$.
We denote these variables by $x_{i_1}, \ldots, x_{i_s}$. Let $H$ be an extremal graph $H(f_1, x_{i_1}, \ldots, x_{i_s})$
defined in such a way that all its vertices belong to $C_1^* \cup P_0^* \setminus N_0^*$.
Such a graph exists by Claim~\ref{cl:proper_pairs}.
By Lemma~\ref{lem:acyclic} the graph $H$ is acyclic, and hence it has at least $s+1$ vertices. 
Therefore, the set  $C_1^* \cup P_0^* \setminus N_0^*$ has at least $s+1 = k-m$ points.
This conclusion establishes the main result of this section.

\begin{lemma}\label{lem:non-split-without-split-projection}
	Let $f = f(x_1, \ldots, x_n)$ be a positive threshold non-split function 
	with $k$ relevant variables, and for every $i \in [n]$ at least one of the restrictions 
	$f_0 = f_{|x_i=0}$ and $f_1 = f_{|x_i=1}$ is non-split.
	Then the number of extremal points of $f$ is at least $k+2$.
\end{lemma}


\section{Conclusion and open problems}
\label{sec:con}

In this paper we studied the cardinality and structure of two sets related to teaching
positive threshold Boolean functions: the specifying set and the set of their extremal
points.

First,  we showed the existence of positive threshold Boolean functions of $n$ variables,
which are not linear read-once and for which the {\it specification number} is at its lowest
bound, $n + 1$ (Theorem~\ref{prop:1}). An important open problem is to describe the set of all
such functions.

Second, we completely  described the set of all positive threshold Boolean functions of $n$ relevant
variables, for which the number of {\it extremal points} is at its lowest bound, $n + 1$. 
This is precisely the set of all positive linear read-once functions (Theorem~\ref{th:extremal_main}). 
It would be interesting to find out whether this result is valid for all positive functions, not necessarily threshold.
In other words,  is it true that a positive Boolean function of $n$ relevant
variables has $n + 1$ extremal points if and only if it is linear read-once?

Finally, we ask whether the acyclic structure of the set of extremal points of a positive threshold function $f$ can be helpful in determining 
the specification number of $f$.

\paragraph*{Acknowledgment} This work was supported by the Russian Science Foundation Grant No. 17-11-01336.

\end{document}